\documentclass[12pt, oneside]{amsart}
\usepackage[all]{xy}
\usepackage{amsfonts}
\usepackage{amsmath,amssymb, amsthm,amsxtra}
\usepackage{verbatim}  %%% comment environment
\usepackage{graphicx}

\newcommand{\Spec} {\operatorname{Spec}}
\newcommand{\coker} {\operatorname{coker}}
\newcommand{\rk} {\operatorname{rk}}
\usepackage[mathscr]{eucal}
\renewcommand{\mathcal}{\mathscr}
\title[Higher Gauss Maps of Veronese Varieties]%
{Higher Gauss Maps of Veronese Varieties \\ 
---a generalization of Boole's formula and \\
degree bounds for higher Gauss map images---}

\author{Hajime Kaji}
\thanks{%
Department of Mathematics, School of Science and Engineering, 
Waseda University, 
\\
\indent
Ohkubo 3-4-1, Shinjuku, Tokyo 169--8555, JAPAN. 
\\
\indent
\textit{E-mail address}: \texttt{kaji@waseda.jp}%
}

\address{% 
Department of Mathematics, School of Science and Engineering, 
Waseda University \\
\indent
Ohkubo 3-4-1, Shinjuku, Tokyo 169--8555, JAPAN}
\email{kaji@waseda.jp}
\subjclass[2000]{%
14N05%Projective techniques
}
\keywords{%
higher Gauss map, Veronese variety, dual variety, Boole's formula, 
degree bound%asymptotic behavior%
}

\theoremstyle{plain}
	\newtheorem{theorem}{Theorem}[section]
	\newtheorem*{theorem*}{Theorem}
	\newtheorem{corollary}[theorem]{Corollary}
	\newtheorem{proposition}[theorem]{Proposition}
	\newtheorem{lemma}[theorem]{Lemma}

	\newtheorem{conjecture}[theorem]{Conjecture}
	\newtheorem{fact}[theorem]{Fact}
\theoremstyle{definition}
	
	\newtheorem{example}[theorem]{Example}
	\newtheorem{observation}[theorem]{Observation}
\theoremstyle{remark}
	\newtheorem{remark}[theorem]{Remark}
\numberwithin{equation}{section}
\newcommand{\kaigyouhaba}{6pt}

\begin{document}

\begin{abstract} 
The image of the higher Gauss map for a projective variety is discussed. 
The notion of higher Gauss maps here 
was introduced by Fyodor L. Zak 
as a generalization of both ordinary Gauss maps and conormal maps. 
The main result is a closed formula for the degree of those images of Veronese varieties.  
This yields a generalization of 
a classical formula by George Boole 
on the degree of the dual varieties of Veronese varieties in 1844. 
As an application of our formula, 
degree bounds for 
higher Gauss map images
of Veronese varieties are given. 
\end{abstract}

\maketitle

\setcounter{section}{-1}

\section{Introduction}

In 1844, 
George Boole proved a beautiful theorem 
that the locus of singular hypersurfaces of degree $d$ in $\mathbb P^n$ 
is a hypersurface 
of degree $(n+1)(d-1)^n$ 
in the linear system $\vert \mathcal O_{\mathbb P^n}(d) \vert$ 
under the assumption that the characteristic of the ground field is equal to zero 
(\cite[p.19]{boole1841} and \cite[p.171]{boole1844}). 
This hypersurface 
is nothing but 
the dual variety 
$X^* \subseteq \check{\mathbb P}^N$ 
of the Veronese variety $X=v_d(\mathbb P^n) \subseteq \mathbb P^N$ 
via 
$
\vert \mathcal O_{\mathbb P^n}(d) \vert
=
\check{\mathbb P}^N
$ with $N+1 = \binom{n+d}{d}$, 
as it is easily shown, 
where $v_d$ is the $d$-fold Veronese embedding $\mathbb P^n \hookrightarrow \mathbb P^N$, 
and 
the dual variety of $X$ is 
by definition the variety of tangent hyperplanes to $X$ 
in the dual projective space $\check{\mathbb P}^N$ of $\mathbb P^N$. 
His theorem is thus stated also as follows:

\smallskip 

\noindent
\textbf{Boole's Formula.} 
\textit{%
The degree of the dual variety of  
the Veronese variety $X = v_d(\mathbb P^n)$ is given by
}
\begin{equation}\label{equation:Boole's_formula}
\deg X ^* = (n+1)(d-1)^n . 
\end{equation} 

\smallskip

The purpose of this article is to generalize Boole's formula 
to the case of the variety of linear spaces tangent to $X$ 
instead of the dual variety, 
where we put the generalization 
introduced by Fyodor L. Zak \cite{zak1987} in context.

For a non-degenerate projective variety $X \subseteq \mathbb P^N$ of dimension $n$
over an algebraically closed field $k$ of arbitrary characteristic $p$, 
set 
$$
Z_m := 
\overline{%
\coprod_{x \in X\text{ : non-singular}} 
\{(x, W) \vert T_xX \subseteq W (\subseteq \mathbb P^N)%
\}
}
\subseteq X \times  \mathbb G(m+1, V) , 
$$
and denote by 
$\gamma_m : Z_m \to \mathbb G(m+1, V)$ the projection, 
where 
$\mathbb G(m+1, V)$ is the Grassmann variety 
parametrizing  
$(m+1)$-quotient spaces of $V:= H^0(\mathbb P^N , \mathcal O_{\mathbb P^N}(1))$, 
or equivalently, 
parametrizing $m$-planes in $\mathbb P^N$. 
The morphism 
$\gamma_m : Z_m \to \mathbb G(m+1, V)$ 
is 
called the \textit{$m$-th Gauss map} of $X\subseteq \mathbb P^N$,  
and its image 
$X_m^* := \gamma_m (Z_m)$ in $\mathbb G(m+1, V)$ 
is called 
the \textit{variety of $m$-dimensional tangent spaces} to $X$ 
(\cite{zak1987}, \cite{zak1993}).
Note that 
$X_{m}^{*}$ is a projective variety in $\mathbb P^{\binom{N+1}{m+1}-1}=\mathbb P(\wedge^{m+1}V)$ 
via the Pl\"ucker embedding $\mathbb G_{X}(m+1,V) \hookrightarrow \mathbb P(\wedge^{m+1}V)$ 
(with reduced structure). 
Note also that $\gamma_{n}$ is the ordinary Gauss map of $X$, and 
that $X_{N-1}^*=X^{*}$ is the dual variety, 
$\gamma_{N-1}$ is the conormal map, and 
$Z_{N-1} = \mathbb P_{X}(\mathcal N)$ is the conormal variety, 
where $\mathcal N$ denotes 
the twisted normal bundle 
$N_{X/\mathbb P^{N}} \otimes \mathcal O_{X}(-1)$ 
of $X$ in $\mathbb P^{N}$
(\cite{fulton}, \cite{katz}, \cite{kleiman1977}, \cite{tevelev}, \cite{zak1987}, \cite{zak1993}).

Our main result is

\begin{theorem}\label{theorem:main_theorem}
For the Veronese variety $X=v_d(\mathbb P^n)
\subseteq \mathbb P^{N}$ with $N+1 = \binom{n+d}{d}$, 
the $m$-th Gauss map $\gamma_m$ is generically finite 
and $\dim X_{m}^{*} = n+(N-m)(m-n)$ 
for each $m$ with $n \le m \le N-1$. 
Moreover if the characteristic $p$ 
of $k$ is equal to zero, then 
the degree of $X_{m}^{*}$ 
in $\mathbb P^{\binom{N+1}{m+1}-1}$ is given by 
\begin{align*}
\deg X_{m}^* 
&=
\frac{1}{n!(n+1)^{n}}
\deg X_{n}^{*} 
\sum_{ \vert \lambda \vert = n} 
f^{\lambda}f^{\lambda + \varepsilon}
\prod_{i=1}^{N-m}
\frac{(n+i)! }{(n+i-\lambda_i)! }
\end{align*}
with $\deg X_{n}^{*} = (n+1)^{n}(d-1)^{n}$, 
where $\lambda =(\lambda_1, \dots , \lambda_{N-m})$ is a partition 
with $\vert \lambda \vert := \sum_{i=1}^{N-m}\lambda_i$, 
$\varepsilon := ((m-n)^{N-m})$ is $N-m$ copies of the integer $m-n$, 
and 
$f^{\lambda}$ is the number of standard Young tableaux with shape $\lambda$. 
\end{theorem}

Note that 
the number of standard Young tableaux above is 
known to be 
equal to the dimension of the corresponding irreducible representation of 
the symmetric group, and is 
computed easily from the Young diagram by virtue of 
the Hook Length Formula 
by Frame, Robinson and Thrall 
(\cite[p.53, Exercise 9]{fulton-young};  \eqref{equation:hook_lengths}).

The assumption $p=0$ 
above is essential, 
while 
the equality above holds for $m=n$ in arbitrary characteristic 
except for the case $p = 2$ and $n=1$ 
(Proposition \ref{proposition:ordinary_gauss_map}
\eqref{proposition:ordinary_gauss_map_Veronese}).

As an application of Theorem \ref{theorem:main_theorem}, 
we obtain bounds for the degree of $X_{m}^{*}$, 
as follows:

\begin{corollary}[Cf.~Heier-Takayama {\cite[Theorem 1.1]{heier-takayama}}]
\label{corollary:degree_approximation}
For the Veronese variety $X=v_d(\mathbb P^n)
\subseteq \mathbb P^{N}$ with 
$N+1 = \binom{n+d}{d}$ in characteristic $p=0$ 
and for each $m$ with $n \le m \le N-1$, 
we have 
\begin{equation*}
\tbinom{N-m}{n}/\tbinom{N-n}{n}
\le
\deg X_{m}^* 
\left/ 
\left[ 
\tbinom{n+ \dim G}{n}\deg G
\deg X_{n}^* 
\right]\right. 
\le 
\tbinom{N-m+n-1}{n}/\tbinom{N-1}{n}
,
\end{equation*}
where $G:=\mathbb G(m-n,N-n)$ is 
the Grassmann variety parametrizing 
$(m-n)$-quotient spaces of an $(N-n)$-dimensional vector space, 
and $\tbinom{k}{n}:= 0$ if $k<n$.
\end{corollary}

We will show that the equalities 
in the above hold 
for arbitrary 
non-degenerate, non-singular projective curves $X$ in $\mathbb P^N$ 
(Corollary \ref{corollary:degree_formula_in_dim1}).

\smallskip

The article is organized as follows: 
In \S1 we show basic results: 
in particular, we show that 
$Z_m$ is isomorphic to Grassmann bundles 
$\mathbb G_{X}(N-m,\mathcal N) 
\simeq 
\mathbb G_{X}(m-n,\mathcal N^{\vee})$
over $X$,  
and study the pull-back of 
the line bundle $\gamma_{m}^{*}\mathcal O_{\mathbb G(m+1,V)}(1)$ on $Z_{m}$
to those Grassmann bundles 
(Lemma \ref{lemma:tautological_line_bundle_small_grassmann}). 
In particular, it turns out that 
$Z_m$ is a Grassmann bundle over the image of 
the ordinary Gauss map $\gamma_{n}$ 
(Remark \ref{remark:ordinary_Gauss_map}). 
In \S2 we show a degree formula of  $X_m^*$ 
for Veronese curves $X=v_d(\mathbb P^1) \subseteq \mathbb P^d$ 
(Theorem \ref{theorem:degree_formula_Veronese_curve}), 
and for projective varieties $X$ in general,  
in terms of the Schur polynomials 
$\varDelta_{\lambda}(s(\mathcal N))$ 
in the Segre classes of $\mathcal N$ 
(Theorem \ref{theorem:degree_gamma_m}). 
The latter result is a generalization of  
the formula of Katz-Kleiman 
(\cite[\S5]{katz}, \cite[IV, \S D]{kleiman1977}; 
Corollary \ref{corollary:degree_gamma_N-1_general}), 
and is used in the proof of the main result. 
Then, 
we study the Schur polynomials $\varDelta_{\lambda}(s(\mathcal N))$ 
for Veronese varieties 
(Proposition \ref{proposition:Schur_of_C_dual}),  
and 
show 
Theorem \ref{theorem:main_theorem} 
and Corollary \ref{corollary:degree_approximation}. 
In addition, we give 
a combinatorial identity on $f^{\lambda}$
(Corollary \ref{corollary:f_lambda_formula}),  
another form of the degree formula in the main result 
(Theorem \ref{theorem:main_theorem'}), 
and 
a conjecture related to 
Corollary \ref{corollary:degree_approximation} 
(Conjecture \ref{conjecture:degree_approximation}). 
Finally, we give explicit formulae of $\deg X_m^*$ 
for Veronese varieties of low dimensions 
(Theorems 
\ref{theorem:degree_formula_Veronese_surface}, 
\ref{theorem:degree_formula_Veronese_3fold}).

\section{Set-up}

Let $X$ be a projective variety of dimension $n$ in $\mathbb P^{N}$
defined over an algebraically closed field $k$ of arbitrary characteristic $p$. 
Throughout this article, we assume that 
$X$ is \textit{non-degenerate} and \textit{non-singular} with $X \ne \mathbb P^N$.  
We have the following commutative diagram, 
with the same notation as in Introduction:
\begin{figure}[ht]
$$
\xymatrix@C=24pt@R=48pt{%
& 
Z_m 
\ar@{^{}->>}[r] 
\ar@{^{}->>}[dr] ^(.4){\gamma_m} 
\ar@{^{(}->}[d] _{} 
\ar@{^{}->}[dl] _{\pi_m} 
& 
\gamma_m(Z_m)=: X_m^* 
\ar@{^{(}->}[d] _{} 
\\ %%%%%%%%%%%%%%%%%%%%%%%%%%%%%%%%%%%%%%%
X & 
 X \times  \mathbb G(m+1, V) 
\ar@{^{}-{>}}[r] _{} 
\ar@{^{}-{>}}[l] _{} 
& 
\mathbb G(m+1, V) 
\ar@{^{(}->}[r] _{} 
& 
\mathbb P(\wedge^{m+1} V)
, 
}
$$
\end{figure}

\noindent 
where 
$Z_m := \coprod_{x \in X} \{  (x, W) \vert T_xX \subseteq W  \}$
and $\pi_m$ is the projection.  
Set 
$$
d_m^* 
:=
\int _{Z_m} c_1(\gamma_m^* \mathcal O_{\mathbb G(m+1,V)}(1)) ^{\dim Z_m} , 
$$
which is non-negative since $\mathcal O_{\mathbb G(m+1,V)}(1)$ is very ample, 
where $\int_X \cdots$ denotes the degree of $\cdots$ as a zero-cycle.

\begin{lemma}\label{lemma:deg_X_m}
We have 
\begin{equation} \label{equation:deg_X_m}
d_m^* 
=\begin{cases}
\deg \gamma_m \cdot \deg X_{m}^{*} , & 
\textsl{if $\gamma_{m}$ is generically finite}
, \\ 
0 , & \textsl{otherwise} ,  
\end{cases} 
\end{equation}
and 
if $d_{m}^{*}>0$ and $p=0$, 
then 
$\deg X_{m}^{*} = d_{m}^{*}$. 
\end{lemma}

\begin{proof}
We have 
$$
d_m^* =\deg \gamma_m \cdot 
\int _{X_m^*} c_1(\mathcal O_{\mathbb G(m+1,V)}(1)\vert_{X_m^*})^{\dim Z_m}
, 
$$ 
where 
$\deg \gamma_m$ 
is by definition 
equal to the degree of the extension of function fields $K(Z_{m})/K(X_{m})$ 
if 
$\gamma_{m}$ is generically finite, 
and equal to zero otherwise 
(\cite[I, \S2, Proposition 6]{kleiman1966}). 
Therefore, 
$d_{m}^{*} > 0$ if and only if $\gamma_{m}$ is generically finite, 
and \eqref{equation:deg_X_m} follows: Indeed, in that case we have  
$
\int _{X_m^*} c_1(\mathcal O_{\mathbb G(m+1,V)}(1)\vert_{X_m^*})^{\dim Z_m} =\deg X_{m}^{*}
.
$

The latter assertion follows from \eqref{equation:deg_X_m}: 
Indeed, if $p=0$ and $\gamma_{m}$ is generically finite, 
then $\deg \gamma_{m} = 1$ by a theorem of Zak 
\cite[Theorem 5]{zak1987}. 
\end{proof}

\begin{remark}
It might be natural to ask if 
a general fibre of $\gamma_m$ is linear also in $p>0$
under the assumption that $\gamma_m$ is separable.  
For the case $m = N-1$, 
this is so by 
the Monge-Segre-Wallace criterion 
\cite[(4.4)]{kleiman1984}: 
Indeed, if $\gamma_{N-1}$ is separable, then 
$X$ is reflexive, so that 
$Z_{N-1} \to X_{N-1}^*$ is 
a projective space bundle over the non-singular locus of $X_{N-1}^*$. 
For the case $m=n$, 
the answer is affirmative: this is a theorem of K. Furukawa 
\cite[Theorem 1.1]{furukawa}. 
\end{remark}

Denote 
by $\mathcal P$ 
the bundle
$P_X^1(\mathcal O_X(1))$
 of principal parts of 
$\mathcal O_X(1) := \mathcal O_{\mathbb P^N}(1)\vert _X$ of first order on $X$ (\cite{piene}), 
and by $\mathcal N$ 
the 
normal bundle 
$N_{X/\mathbb P^N}$ 
of $X$ in $\mathbb P^N$ twisted by $\mathcal O_X(-1)$: 
$$
\mathcal P := P_X^1(\mathcal O_X(1)) 
,
\qquad 
\mathcal N := N_{X/\mathbb P^N} \otimes \mathcal O_X(-1)  . 
$$ 
In fact, 
since $X$ is non-singular, 
these are 
both locally free as sheaves on $X$ with rank $n+1$ and $N-n$, 
respectively,   
We have a commutative diagram 
of vector bundles on $X$ 
with exact rows and columns, 
as follows: 
\begin{equation}\label{diagram:principal_bundle}
\begin{matrix}
& & 0 & & 0\\[\kaigyouhaba] 
& & \downarrow& & \downarrow & & \\[\kaigyouhaba] 
& & \mathcal I_{X/\mathbb P^N} / \mathcal I_{X/\mathbb P^N}^2 \otimes \mathcal O_X(1) 
& = & \mathcal N^\vee  \\[\kaigyouhaba] 
& & \downarrow& & \downarrow & & \\[\kaigyouhaba]  
0 & \to & \Omega_{\mathbb P^N}^1
\vert _X
\otimes \mathcal O_X(1) & \to & V \otimes \mathcal O_X & \to &   \mathcal O_X(1) & \to & 0 \\[\kaigyouhaba] 
& & \downarrow & & \downarrow & & \vert \hskip -1pt \vert  \\[\kaigyouhaba] 
0 & \to & \Omega_X^1\otimes \mathcal O_X(1) & \to &  \mathcal P & \to &  \mathcal O_X(1) & \to  & 0\\[\kaigyouhaba] 
& & \downarrow& & \downarrow & & \\[\kaigyouhaba] 
& & 0 & & 0 
\end{matrix} 
\end{equation}

\begin{remark}\label{remark:ordinary_Gauss_map}
The ordinary Gauss map 
$\gamma_n$ 
is a morphism  
from $X \simeq Z_n$ 
to the Grassmann variety $\mathbb G(n+1, V)$ 
induced from 
the quotient 
$V\otimes \mathcal O_X \to \mathcal P$ above 
by the universality of $\mathbb G(n+1, V)$: 
In fact, 
$\mathbb P(\mathcal P\otimes k(x)) = T_{x}X$ 
in $\mathbb P(V) = \mathbb P^{N}$ for any $x\in X$.  
Hence
$V\otimes \mathcal O_X \to \mathcal P$ 
is equal to the pull-back 
by $\gamma_n$
of the universal quotient on 
$\mathbb G(n+1, V)$ (\cite[IV, \S A]{kleiman1977}).
\end{remark}

\begin{proposition}\label{proposition:ordinary_gauss_map}
\begin{enumerate}
\item \label{proposition:ordinary_gauss_map_general}
If $p=0$, or more generally if $\gamma_{n}$ is separable, 
then 
$$\deg X_{n}^{*}= \int_{X} (K_{X}+(n+1)H)^{n} , $$
where $K_{X}$ is the canonical divisor class and 
$H$ is the class of hyperplane sections of $X$, that is, 
$H := c_{1}(\mathcal O_{X}(1))$. 
\item \label{proposition:ordinary_gauss_map_Veronese}
For the Veronese variety $X=v_{d}(\mathbb P^{n})$, 
if $p\ne 2$ or $n \ge 2$, then 
$$
\deg X_{n}^{*} = (n+1)^{n}(d-1)^{n} 
,
$$
and if $p=2$ and $n=1$, then $\deg X_1^* = d-1$. 
\end{enumerate}
\end{proposition}

\begin{proof}
(1) 
We have that 
$\gamma_{n}^{*} \mathcal O_{\mathbb G(n+1,V)}(1) = \det \mathcal P$ 
with $X\simeq Z_{n}$ (Remark \ref{remark:ordinary_Gauss_map})
and 
\begin{equation}\label{equation:c_1_of_principal_bundle} 
c_{1}(\det \mathcal P)
=K_{X}+(n+1)H
\end{equation}
by \eqref{diagram:principal_bundle}, 
hence $d_{n}^{*} = \int_{X} (K_{X}+(n+1)H)^{n} .$
According to a theorem of Zak 
\cite[Theorem 5]{zak1987}, 
$\gamma_{n}$ is finite since $X$ is non-singular. 
Thus the equality,  
$d_{n}^{*} = \deg X_{n}^{*}$
follows from 
Lemma \ref{lemma:deg_X_m} if  $p = 0$, 
or 
from \eqref{equation:deg_X_m} and 
a theorem of Furukawa \cite[Theorem 1.1]{furukawa} if $\gamma_{n}$ is separable.

\noindent 
(2) 
For $X = v_{d}(\mathbb P^{n})$ we have 
$K_{X}+(n+1)H = (-n-1)h + (n+1)dh
=(n+1)(d-1)h$, 
where $h := c_{1}(\mathcal O_{\mathbb P^{n}}(1))$.
If $p\ne 2$ or $n \ge 2$, then 
the Gauss map $\gamma_{n}$ for a Veronese variety
is an embedding (\cite{kajiRGMVV}, \cite{kajiSGMR}), 
hence the conclusion follows.  
On the other hand, 
an explicit calculation shows that if $p=2$ and $n=1$, then $X_1^*$ is 
a non-singular rational curve of degree $d-1$ 
in $\mathbb P(\wedge ^2 V)\simeq \mathbb P^{\binom{d+1}{2}-1}$: 
In fact, 
$X_1^*$ is equal to 
the Veronese curve of degree $d-1$ in 
the linear subspace spanned by itself $X_1^*$ 
with respect to suitable homogeneous coordinates,  
and 
the extension $K(Z_{1})/K(X_{1}^{*})$ defined by 
$\gamma_1$ is purely inseparable of degree $2$. 
\end{proof}

For a vector bundle $\mathcal E$ of rank $r$ on $X$ and an integer $d$ with 
$0 < d < r$, %$0 \le d \le r$, 
let $\mathbb G_X(d,\mathcal E)$ be 
the Grassmann bundle of $\mathcal E$ over $X$ 
parametrizing corank $d$ subbundles of $\mathcal E$. 
%where $\mathbb G_X(d,\mathcal E) = X$ for $d = 0, r$. 
Denote simply by $\mathbb G(d,r)$ the Grassmann variety $\mathbb G_{\Spec k}(d,k^{\oplus r})$
when 
$X = \Spec k$. 
Let $\pi: \mathbb G_X(d,\mathcal E) \to X$ be the projection, 
let 
$ \mathcal Q \gets \pi^*\mathcal E$ 
be the universal quotient bundle of rank $d$ on $\mathbb G_X(d,\mathcal E)$, 
and denote by 
$\mathcal O_{\mathbb G_X(d,\mathcal E)}(1)$ 
the pull-back of the tautological line bundle 
$\mathcal O_{\mathbb P_X(\wedge^d \mathcal E)}(1)$ 
by the relative 
Pl\"ucker embedding, 
$\mathbb G_X(d,\mathcal E)
\hookrightarrow \mathbb P_X(\wedge^d \mathcal E)$ over $X$.  
Then we have $\mathcal O_{\mathbb G_X(d,\mathcal E)}(1) = \det \mathcal Q$, 
since the embedding above is defined by the quotient
$\wedge^d \pi^*  \mathcal E \to \wedge^d \mathcal Q=\det \mathcal Q$.

\begin{fact}[{\cite{kaji-terasoma}}]\label{fact:degree_formula_for_grassmann_bundles}
Let 
$\pi_{*} : A^{*+ (r-d)d}(\mathbb G_X(d,\mathcal E)) \to A^{*}(X)$ 
be 
the push-forward between the Chow rings 
induced by $\pi$. 
Then for each integer $k \ge (r-d)d$, we have 
$$
\pi_{*}( 
c_1(\mathcal O_{\mathbb G_X(d,\mathcal E)}(1))^{k}
)
= 
\sum_{ \vert \lambda \vert = k - (r-d)d} 
f^{\lambda+\varepsilon }
\varDelta _{\lambda}(s(\mathcal E))   
, 
$$
where 
$\lambda =(\lambda_1 , \dots, \lambda_d)$
is a partition 
with $\vert \lambda \vert = \sum \lambda_i$, 
$\varepsilon := ((r-d)^d)$ is $d$ copies of the integer $r-d$, 
$f^{\lambda+\varepsilon}$
denotes the number of standard Young tableaux with shape 
${\lambda+\varepsilon}$, 
and 
$\varDelta_{\lambda}(s(\mathcal E))
:= \det[s_{\lambda_{i}+j-i}(\mathcal E)]_{1 \le i,j \le d}$ 
is the Schur polynomial in the Segre classes of $\mathcal E$ for $\lambda$. 
\end{fact}

In particular, 
the degree of the Grassmann variety $\mathbb G(d, r)$ 
is equal to $f^{\varepsilon}$ 
with respect to the Pl\"ucker embedding,  
as it is classically well known ({\cite[Example 14.7.11 (iii)]{fulton}}). 
Thus we have 
\begin{equation}\label{equation:degree_of_Grassmann_variety}
\deg \mathbb G(d, r) 
= f^{\varepsilon} 
, \quad 
\dim \mathbb G(d, r) 
=\vert \varepsilon \vert 
\quad 
(\varepsilon := ((r-d)^d))
. 
\end{equation}
Note that 
the number of standard Young tableaux above is 
given explicitly, as follows
(\cite[p.54, Exercise 9]{fulton-young}): 
\begin{equation}\label{equation:hook_lengths}
f^{\lambda} 
= \frac
{\vert \lambda \vert !\prod_{1 \le i < j \le d} (\lambda_i -\lambda_j+j-i)}
{\prod_{1 \le i \le d} (\lambda_i+d-i)!}
.
\end{equation}

The Segre classes $s_i(\mathcal E)$ here are
the ones satisfying
$s(\mathcal E)c(\mathcal E^\vee)=1$ 
as in \cite{kleiman1977}, 
where 
$s(\mathcal E)$ and $c(\mathcal E)$ are respectively 
the total Segre class and the total Chern class of $\mathcal E$. 
Note that our Segre class $s_i(\mathcal E)$ 
differs by the sign $(-1)^i$ 
from the one in \cite{fulton}.

\begin{lemma}\label{lemma:tautological_line_bundle_twisted}
\begin{enumerate}
\item \label{lemma:tautological_line_bundle_twisted_1}
If $\mathcal L$ is a line bundle on $X$, then 
there is an isomorphism, 
$ i : \mathbb G_X(d,\mathcal E) \to \mathbb G_X(d,\mathcal E \otimes \mathcal L) $
over $X$ such that 
\begin{equation*}
i^*\mathcal O_{\mathbb G_X(d,\mathcal E\otimes \mathcal L)} (1) 
= 
 \mathcal O_{\mathbb G_X(d,\mathcal E)} (1) 
\otimes \pi^* \mathcal L^{\otimes d} . 
\end{equation*}
\item \label{lemma:tautological_line_bundle_twisted_2}
In particular, 
if $\mathcal E=\mathcal L^{\oplus r}$, 
then there is an isomorphism, 
$
i : 
\mathbb G(d, r) \times X 
\to 
\mathbb G_X(d,\mathcal L^{\oplus r})
$
such that 
$$
i^* \mathcal O_{\mathbb G_X(d,\mathcal L^{\oplus r})} (1) 
=p_{1}^*  \mathcal O_{\mathbb G(d, r)}(1) \otimes p_{2}^* \mathcal L^{\otimes d} , 
$$
where 
$p_{1}: \mathbb G(d, r) \times X \to \mathbb G(d, r)$
and $p_{2}: \mathbb G(d, r) \times X \to X$
are the projections. 
\end{enumerate}
\end{lemma}

\begin{proof}
(1) 
Let 
$\mathcal Q \gets \pi^*\mathcal E$ and 
$\mathcal Q' \gets \pi'{}^*(\mathcal E\otimes \mathcal L)$ 
be respectively 
the universal quotient bundles 
of 
$\mathbb G_X(d, \mathcal E)$ and of 
$\mathbb G_X(d, \mathcal E \otimes \mathcal L)$, 
where $\pi' : \mathbb G_X(d, \mathcal E \otimes \mathcal L)  \to X$ is the projection. 
By the universality of $  \mathbb G_X(d, \mathcal E \otimes \mathcal L)$ 
the quotient $[\pi^*\mathcal E \to \mathcal Q]\otimes \pi^{*} \mathcal L$ on $\mathbb G_X(d, \mathcal E)$  
induces a morphism 
$i : \mathbb G_X(d,\mathcal E) \to  \mathbb G_X(d, \mathcal E \otimes \mathcal L)$
such that 
$[\pi^*\mathcal E \to \mathcal Q]\otimes \pi^{*} \mathcal L 
=i^{*}[\pi'{}^*(\mathcal E\otimes \mathcal L) \to \mathcal Q']$. 
Then we have 
$$
i^{*}\mathcal O_{\mathbb G_X(d, \mathcal E \otimes \mathcal L)}(1) 
=i^{*} \det \mathcal Q'
=\det (\mathcal Q \otimes \pi^{*}\mathcal L)
=\det \mathcal Q \otimes \pi^{*}\mathcal L^{\otimes d} 
=\mathcal O_{\mathbb G_X(d,\mathcal E)}(1)
 \otimes \pi^{*}\mathcal L^{\otimes d} 
. 
$$
Similarly to the above, 
the quotient 
$[\pi'{}^*(\mathcal E\otimes \mathcal L) \to 
\mathcal Q'] \otimes \pi'{}^{*}\mathcal L^{\vee}$ 
on $\mathbb G_X(d, \mathcal E \otimes \mathcal L)$ 
induces a morphism 
$i' : \mathbb G_X(d, \mathcal E \otimes \mathcal L) \to \mathbb G_X(d,\mathcal E)$ 
such that 
$[\pi'{}^*(\mathcal E\otimes \mathcal L) \to \mathcal Q'] \otimes \pi'{}^{*}\mathcal L^{\vee}
= i'{}^{*}[\pi^*\mathcal E \to \mathcal Q]$. 
Then it follows 
from the universality of those Grassmann bundles 
that $i\circ i' = 1_{\mathbb G_X(d, \mathcal E \otimes \mathcal L)}$ 
and $i'\circ i = 1_{\mathbb G_X(d,\mathcal E)}$. 

\noindent (2) 
Apply the result (1) above to the case $\mathcal E = \mathcal O_X^{\oplus r}$ with 
$\mathbb G_X(d , \mathcal O_X^{\oplus r})
= \mathbb G(d,r) \times X$. 
\end{proof}

\begin{lemma}\label{lemma:tautological_line_bundle_small_grassmann}
\begin{enumerate}
\item \label{lemma:tautological_line_bundle_small_grassmann_C}
There is an isomorphism, 
$j : Z_m \to \mathbb G_X(m-n, \mathcal N^\vee)$
over $X$ 
such that 
\begin{equation*}
j_{*}\gamma_m ^*\mathcal O_{\mathbb G(m+1,V)}(1) 
= \mathcal O_{\mathbb G_X(m-n, \mathcal N^\vee)}(1) 
\otimes \pi_{m}^* \det \mathcal P , 
\end{equation*}
where the projection $\mathbb G_X(m-n, \mathcal N^\vee)\to X$ is denoted also by 
$\pi_{m}$. 
\item \label{lemma:tautological_line_bundle_small_grassmann_C_dual}
Let 
$j' : Z_m \to \mathbb G_X(N-m, \mathcal N )$ be 
the isomorphism 
over $X$ 
defined by the composition of $j$ and the canonical 
isomorphism $\mathbb G_X(m-n, \mathcal N^\vee) \simeq \mathbb G_X(N-m, \mathcal N)$. 
Then we have 
\begin{equation*}
j'_{*}\gamma_m ^*\mathcal O_{\mathbb G(m+1,V)}(1) = 
\mathcal O_{\mathbb G_X(N-m, \mathcal N )}(1) .
\end{equation*}
\end{enumerate}
In particular, 
 $\dim Z_{m} = n+(N-m)(m-n)$.  
\end{lemma}

\begin{proof} (1)
For a scheme $Y$ over $X$, consider 
a commutative diagram of vector bundles 
on $Y$ with 
exact rows and columns, as follows:  
\begin{equation}\label{diagram:vector_bundles_K_W_E}
\begin{matrix}
& & 0 & & 0\\[\kaigyouhaba]
& & \downarrow& & \downarrow & & \\[\kaigyouhaba]
& & \mathcal E & = & \mathcal E \\[\kaigyouhaba]
& & \psi_{}' \downarrow \phantom{\psi_{}'} & & \varphi_{}' \downarrow \phantom{\varphi_{}'} & & \\[\kaigyouhaba]
0 & \to &  \mathcal N^\vee_{Y} & \to & (V \otimes \mathcal O_{X})_{Y} & \to & 
\mathcal P_{Y} & \to & 0 \\[\kaigyouhaba]
& & \psi_{} \downarrow \phantom{\psi_{}} & & \varphi_{} \downarrow \phantom{\varphi_{}} & & \vert \hskip -1pt \vert  \\[\kaigyouhaba]
0 & \to & \mathcal K & \to & \mathcal W & \to &  \mathcal P_{Y}   & \to  & 0\\
& & \downarrow& & \downarrow & & \\[\kaigyouhaba]
& & 0 & & 0 
\end{matrix}
\end{equation}
where
the exact sequence in the middle row is 
the pull-back to $Y$ 
of the one in 
the middle column of \eqref{diagram:principal_bundle} 
(the subscript $Y$ denotes the pull-back to $Y$), 
and 
$\mathcal W$, $\mathcal K$ and $\mathcal E$ 
are of rank $m+1$, $m-n$ and $N-m$, respectively. 
Let 
$\mathcal Q \gets V \otimes \mathcal O_{\mathbb G(m+1, V)}$
be the universal quotient bundle on $\mathbb G(m+1, V)$.

We consider the case when $Y:= Z_{m}$. 
By the definition of $Z_{m}$, 
the quotient, $[V \otimes \mathcal O_{X} \to \mathcal P]_{Z_{m}}$
factors through $\mathcal Q_{Z_{m}}$ 
(Remark \ref{remark:ordinary_Gauss_map}).
Now, set  
$\mathcal W:= \mathcal Q_{Z_{m}}$, 
$\varphi := [V \otimes \mathcal O_{\mathbb G(m+1, V)} \to \mathcal Q]_{Z_{m}}$, 
$\mathcal E:= \ker \varphi$
and 
$\mathcal K:= \ker [\mathcal Q_{Z_{m}} \to \mathcal P_{Z_{m}}]$. 
Let $\psi : \mathcal N_{Z_{m}}^{\vee} \to \mathcal K$ 
be 
the induced homomorphism 
that fits into the diagram \eqref{diagram:vector_bundles_K_W_E}. 
Then the quotient $\psi$ induces a morphism,  
$$
j : Z_{m} \to \mathbb G_{X}(m-n,\mathcal N^{\vee})
$$
such that $\psi$ is equal to the pull-back  
of the universal quotient on $\mathbb G_{X}(m-n,\mathcal N^{\vee})$ 
to $Z_{m}$. 
In particular, we have  
$$
j^{*} \mathcal O_{\mathbb G_{X}(m-n,\mathcal N^{\vee})}(1) = \det \mathcal K 
= \det \mathcal Q_{Z_{m}} \otimes \det \mathcal P_{Z_{m}} ^{\vee}
= \mathcal O_{\mathbb G(m+1, V)}(1)_{Z_{m}} \otimes \det \mathcal P_{Z_{m}} ^{\vee}
.
$$

To show that $j$ is an isomorphism, 
we consider the case when 
$Y:= \mathbb G_{X}(m-n,\mathcal N^{\vee})$, 
$\psi'$ and $\psi$ are respectively 
the universal subbundle and the universal quotient bundle of 
$\mathbb G_{X}(m-n,\mathcal N^{\vee})$. 
Define $\varphi'$ to be the composition of $\psi'$ and 
$[\mathcal N^{\vee} \to V \otimes \mathcal O_{X}]_{\mathbb G_{X}(m-n,\mathcal N^{\vee})}$, 
set $\mathcal W:= \coker \varphi'$, and let 
$\varphi$ 
be the canonical quotient.  
Then the morphism 
$\mathbb G_{X}(m-n,\mathcal N^{\vee})\to X \times \mathbb G(m+1, V)$ 
induced by $\varphi$ 
gives the inverse of $j$.

\noindent (2) 
This follows from
(1) and 
 Lemma \ref{lemma:tautological_line_bundles_via_canonical_isomorphism} below, 
since $\det \mathcal N = \det \mathcal P$ by \eqref{diagram:principal_bundle},  
while one can prove (2) directly by a similar way to (1). 
\end{proof}

\begin{lemma}\label{lemma:tautological_line_bundles_via_canonical_isomorphism}
For a vector bundle $\mathcal E$ of rank $r$ on a variety $X$, 
we have an isomorphism, 
$$
\mathcal O_{\mathbb G_X(d,\mathcal E)}(1) 
\simeq 
\mathcal O_{\mathbb G_X(r-d,\mathcal E^\vee)}(1) \otimes \pi^*\det \mathcal E, 
$$
via the canonical isomorphism, 
$\mathbb G_X(d,\mathcal E)\simeq \mathbb G_X(r-d,\mathcal E^\vee)$.  
\end{lemma}

\begin{proof}
Let $\mathcal S \rightarrow \pi^{*}\mathcal E$ be 
the universal subbundle of $\mathbb G_X(d,\mathcal E)$, 
and 
let $\mathcal Q \leftarrow \pi^{*}\mathcal E$ be 
the universal quotient bundle of $\mathbb G_X(d,\mathcal E)$, as before. 
Since 
$\mathcal O_{\mathbb G_X(d,\mathcal E)}(1) = \det \mathcal Q$ and
$\mathcal O_{\mathbb G_X(r-d,\mathcal E^\vee)}(1) = \det \mathcal S^{\vee}$, 
the conclusion follows from the isomorphism, 
$\det \mathcal S \otimes \det \mathcal Q \simeq \pi^*\det \mathcal E$. 
\end{proof}

\section{Degree Formulae}
\subsection{Veronese curves}

\begin{theorem}\label{theorem:degree_formula_Veronese_curve}
For the Veronese curve $X = v_d(\mathbb P^1) \subseteq \mathbb P^d$ of degree $d$, 
the $m$-th Gauss map $\gamma_m$ is a finite morphism
 and 
$\dim X_m^* = 1+ (d-m)(m-1)$ 
for each $m$ with $1 \le m \le d-1$. 
Moreover if the characteristic $p$ of $k$ is equal to zero, then 
the degree of $X_{m}^{*}$ 
in $\mathbb P^{\binom{d+1}{m+1}-1}$ is given by 
\begin{align*}
\deg X_m^* 
= \frac{d-m}{d-1} \big(1+\dim \mathbb G(d-m,d-1 )\big) \deg \mathbb G(m-1,d-1) 
\deg X_{1}^{*}
\end{align*}
with $\deg X_{1}^{*} = 2(d-1)$.
\end{theorem}

\begin{proof}
It follows from Lemma 
\ref{lemma:tautological_line_bundle_small_grassmann} 
\eqref{lemma:tautological_line_bundle_small_grassmann_C_dual}  
that 
$$
j'_{*} \gamma_m ^*\mathcal O_{\mathbb G(m+1,V)}(1) 
= 
\mathcal O_{\mathbb G_X(d-m, \mathcal N)}(1) 
. 
$$ 
Since 
$\mathcal N =\mathcal O_{\mathbb P^1}(2)^{\oplus d-1}$ 
in any characteristic 
$p$ 
(\cite[Example (3.5)]{kajiNB}), 
it follows from Lemma \ref{lemma:tautological_line_bundle_twisted} 
\eqref{lemma:tautological_line_bundle_twisted_2}
that 
$$
i^{*}\mathcal O_{\mathbb G_X(d-m, \mathcal N)}(1) 
= 
p_{1}^* \mathcal O_{\mathbb G(d-m, d-1)}(1) \otimes
p_{2}^* \mathcal O_{\mathbb P^1}(2)^{\otimes d-m}
. 
$$ 
Thus we have 
$$
i^{*}j'_{*}\gamma_m ^*\mathcal O_{\mathbb G(m+1,V)}(1) 
= 
p_{1}^* \mathcal O_{\mathbb G(d-m, d-1)}(1) \otimes
p_{2}^* \mathcal O_{\mathbb P^1}(2d-2m) 
, 
$$
which is ample by $m < d$, 
so that 
$\gamma_m ^*\mathcal O_{\mathbb G(m+1,V)}(1)$ is also ample. 
Hence $\gamma_m$ is a finite morphism, 
in particular, $\dim X_{m}^{*} = \dim Z_{m}=1+(d-m)(m-1)$. 
It follows that 
{\allowdisplaybreaks 
\begin{align*}
d_m^{*} 
&= 
\int 
_{\mathbb G(d-m, d-1)\times \mathbb P^1}
\big( p_{1}^*c_1(
\mathcal O_{\mathbb G(d-m, d-1)}(1) )
+ 
p_{2}^*c_1( 
\mathcal O_{\mathbb P^1}(2d-2m)
)
\big)^{M}
\\&=
M \cdot \deg \mathbb G(d-m, d-1) \cdot (2d-2m)
, 
\end{align*}
}%
where 
$M %= \dim X_m^* 
:=1+(d-m)(m-1)$. 
The conclusion follows from Lemma \ref{lemma:deg_X_m}
and 
Proposition \ref{proposition:ordinary_gauss_map}
\eqref{proposition:ordinary_gauss_map_Veronese}. 
\end{proof}

\begin{example}\label{example:Veronese_curve}
Consider the first non-trivial case: 
$X=v_4(\mathbb P^1) \subseteq \mathbb P^4$ with $m=2$.  
According to Theorem \ref{theorem:degree_formula_Veronese_curve}, 
we have $\dim X_2^*=3$, 
and $\deg X_2^* = 12$ if $p=0$, 
where $f^{\varepsilon} = \deg \mathbb G(1,3) = 1$ with $\varepsilon = (1^{2})$.  
This means that 
\textit{%
there are exactly $12$ $2$-planes $W\subseteq \mathbb P^4$ 
such that 
$W$ is tangent to $X$ and 
intersects given $3$ lines in general position}: 
Indeed, 
the class of hyperplane sections
of $\mathbb G(m+1,V) = \mathbb G(3,5)$  
is represented by a divisor,  
$$
\Omega(L)
:= \{ W \vert W \cap L \ne \emptyset \} 
,
$$
where $L \subseteq \mathbb P^4$ is a linear subspace of dimension $N-m-1=1$ 
(\cite[\S14.7]{fulton}). 
\end{example}

\subsection{Projective varieties}

\begin{theorem}\label{theorem:degree_gamma_m}
For a non-degenerate, non-singular projective variety $X\subseteq \mathbb P^N$ 
of dimension $n$ 
and 
for each integer $m$ with $n \le m \le N-1$, 
if the characteristic $p$ of $k$ is equal to zero and $\gamma_{m}$ is generically finite, 
then 
the degree of $X_{m}^{*}$ 
in $\mathbb P^{\binom{N+1}{m+1}-1}$ is given by 
{\allowdisplaybreaks 
\begin{align}
\tag{1}\label{equation:degree_X_m_C_dual}
\deg {X_m^*}
&= 
\sum_{ \vert \lambda \vert = n} 
f^{\lambda+\varepsilon }
\int_X 
\varDelta _{\lambda}(s(\mathcal N))   
,
\\
\tag{2}\label{equation:degree_X_m_C}
\deg {X_m^*}
&= 
\sum_{k = 0}^{n} 
(-1)^{n-k}
\tbinom{M}{k}
\sum_{ \vert \lambda' \vert = n-k} 
f^{\lambda'+\varepsilon' }
\int_X 
\varDelta _{\lambda'}(s(\mathcal N))
\cdot 
(K_{X}+(n+1)H)^{k} 
, 
\end{align}
}%
where 
$\lambda =(\lambda_1 , \dots, \lambda_{N-m})$ 
and 
$\lambda' =(\lambda'_{1} , \dots, \lambda'_{m-n})$ 
are partitions 
with 
$\vert \lambda \vert = \sum \lambda_i$  and 
$\vert \lambda' \vert = \sum \lambda'_i$, 
$\varepsilon := ((m-n)^{N-m})$,  
$\varepsilon' := ((N-m)^{m-n})$, 
$f^{*}$
denotes the number of standard Young tableaux with shape $*$,
$\varDelta_{*}(s(\mathcal N))$ is 
the Schur polynomial in the Segre classes of 
$\mathcal N = { N }_{X/\mathbb P^N} \otimes \mathcal O_X(-1) $ 
for a partition $*$, 
$K_{X}$ is the canonical divisor class, 
$H:= c_{1}(\mathcal O_{X}(1))$ the class of hyperplane sections, 
and 
$M:=\dim X_{m}^{*}=n+(N-m)(m-n)$. 
\end{theorem}

\begin{proof}
\eqref{equation:degree_X_m_C_dual} 
It follows from Lemma 
\ref{lemma:tautological_line_bundle_small_grassmann} \eqref{lemma:tautological_line_bundle_small_grassmann_C_dual}
and 
Fact \ref{fact:degree_formula_for_grassmann_bundles} 
that 
\begin{equation}\label{equation:degree_gamma_m_C_dual}
d_{m}^{*} 
=
\int_{\mathbb G_{X}(N-m, \mathcal N)}
c_1 ( \mathcal O_{\mathbb G_{X}(N-m, \mathcal N) }(1) )^M
=
\sum_{ \vert \lambda \vert = n} 
f^{\lambda+\varepsilon }
\int_X 
\varDelta _{\lambda}(s(\mathcal N))
,
\end{equation}  
where $M= \dim \mathbb G_{X}(N-m, \mathcal N)$. 
Then the conclusion follows from Lemma \ref{lemma:deg_X_m}.

\noindent \eqref{equation:degree_X_m_C} 
It follows from 
Lemma 
\ref{lemma:tautological_line_bundle_small_grassmann} \eqref{lemma:tautological_line_bundle_small_grassmann_C}, 
\eqref{equation:c_1_of_principal_bundle}
and 
Fact \ref{fact:degree_formula_for_grassmann_bundles} 
that 
{\allowdisplaybreaks 
\begin{align*}
d_{m}^{*} 
&=
\int_{\mathbb G_{X}(m-n, \mathcal N)}
(
c_1 ( \mathcal O_{\mathbb G_{X}(m-n, \mathcal N^{\vee}) }(1) )
+ 
\pi_{m}^{*}c_{1}(\det \mathcal P)
)^M
\\&=
\int_{\mathbb G_{X}(m-n, \mathcal N)}
\big(
c_1 ( \mathcal O_{\mathbb G_{X}(m-n, \mathcal N^{\vee}) }(1) )
+ 
\pi_{m}^{*}(K_{X}+(n+1)H)
\big)^M
\\&=
\sum_{k = 0}^{M} 
\tbinom{M}{k}
\int_{\mathbb G_{X}(m-n, \mathcal N)}
c_1 ( \mathcal O_{\mathbb G_{X}(m-n, \mathcal N^{\vee}) }(1) )^{M-k}
\cdot 
\pi_{m}^{*}(K_{X}+(n+1)H)^{k}
\\&=
\sum_{k = 0}^{n} 
\tbinom{M}{k}
\int_X
\pi_{m*}( 
c_1 ( \mathcal O_{\mathbb G_{X}(m-n, \mathcal N^{\vee}) }(1) )^{M-k}
)
\cdot 
(K_{X}+(n+1)H)^{k}
\\&=
\sum_{k = 0}^{n} 
\tbinom{M}{k}
\sum_{ \vert \lambda' \vert = n-k} 
f^{\lambda'+\varepsilon' }
\int_X 
\varDelta _{\lambda'}(s(\mathcal N^{\vee}))
\cdot 
(K_{X}+(n+1)H)^{k}
\\&=
\sum_{k = 0}^{n} 
(-1)^{n-k}
\tbinom{M}{k}
\sum_{ \vert \lambda' \vert = n-k} 
f^{\lambda'+\varepsilon' }
\int_X 
\varDelta _{\lambda'}(s(\mathcal N))
\cdot 
(K_{X}+(n+1)H)^{k}
. 
\end{align*}  
}%
Then the conclusion follows from Lemma \ref{lemma:deg_X_m}. 
\end{proof}

Theorem \ref{theorem:degree_gamma_m}
\eqref{equation:degree_X_m_C_dual}
with $m:= N-1$ yields the Katz-Kleiman formula, as follows:

\begin{corollary}[{Katz \cite[\S5]{katz}, Kleiman \cite[IV, \S D]{kleiman1977}}]
\label{corollary:degree_gamma_N-1_general}
With the same assumption as above, we have 
$$
d_{N-1}^{*} 
= 
\int_X s_n(\mathcal N) , 
$$ 
which is equal to 
$\deg {X_{N-1}^*}$ if $p=0$ and $\gamma_{N-1}$ is generically finite. 
\end{corollary}

Theorem \ref{theorem:degree_gamma_m}
\eqref{equation:degree_X_m_C_dual}
with $n=1$ yields 
a generalization of 
Theorem \ref{theorem:degree_formula_Veronese_curve}: 

\begin{corollary}\label{corollary:degree_formula_in_dim1}
For a non-degenerate, non-singular projective curve 
$X \subseteq \mathbb P^N$ of degree $d$ and genus $g$, 
the $m$-th Gauss map $\gamma_m$ is generically finite and 
$\dim X_m^* = 1+(N-m)(m-1)$ 
for each $m$ with $1 \le m \le N-1$. 
Moreover if $p=0$, %the characteristic $p$ of $k$ is equal to zero, 
then 
\begin{equation*}\label{equation:deg_gamma_m_in_dim1}
\deg X_{m}^* 
= 
\frac{N-m}{N-1 } \big(1+\dim \mathbb G(N-m,N-1)\big)
\deg \mathbb G(N-m,N-1)  
\deg X_{1}^{*}
\end{equation*}
with 
$\deg X_{1}^{*} = 2g-2+2d$
\end{corollary}

\begin{proof}
First of all it follows from Proposition \ref{proposition:ordinary_gauss_map} 
\eqref{proposition:ordinary_gauss_map_general} that 
$\deg X_{1}^{*} = 2g-2+2d$. 
It follows from \eqref{equation:degree_gamma_m_C_dual}
that 
$d_{m}^{*} =(2d+2g-2) f^{\mu}$
with 
$\mu :=  (m,(m-1)^{N-m-1})$: 
Indeed, 
the partition $\lambda$ in the summation 
is determined to be
 $(1, 0^{N-m-1})$ 
by $\vert \lambda \vert = n =1$, 
and 
$\varDelta_{\lambda}(\mathcal N) 
=s_1(\mathcal N)=c_1(\mathcal N)$,  
which is of degree $2d + 2g - 2$ since  
$\det \mathcal N 
=\omega_X \otimes \mathcal O_X(2)$
by \eqref{diagram:principal_bundle}.  
Therefore $d_{m}^{*} >0$ by $X \ne \mathbb P^N$, 
and it follows from 
Lemma \ref{lemma:deg_X_m} that 
$\deg X_{m}^* = (2d+2g-2) f^{\mu}$. 
Now from \eqref{equation:hook_lengths} we see that 
$$
f^{\mu} = \frac{(N-m)(1+(m-1)(N-m))}{N-1} f^{\varepsilon} , 
$$
where $\varepsilon := ((m-1)^{N-m})$, 
so that the conclusion follows from  \eqref{equation:degree_of_Grassmann_variety}.
\end{proof}

\subsection{Veronese varieties}

Let 
$X = v_d(\mathbb P^n) \subseteq \mathbb P^N$ 
be the Veronese variety with $N+1 = \binom{n+d}{d}$, 
and set 
$\mathcal N := { N }_{X/\mathbb P^N} \otimes \mathcal O_X(-1) $, 
as before.
It follows from \eqref{diagram:principal_bundle} and 
the Euler sequence on $\mathbb P^n$
that 
$
s(\mathcal N) 
= 1 / c(\mathcal N^\vee) 
= c(\mathcal P)
=(1+(d-1)h) ^{n+1} , 
$ 
so that 
\begin{equation}\label{equation:Segre_classes_of_C_dual}
s_i(\mathcal N) =  \tbinom{n+1}{i}(d-1)^i h^i,  
\end{equation}
where 
$\mathcal P := P_X^1(\mathcal O_X(1))= P_{\mathbb P^{n}}^1(\mathcal O_{\mathbb P^n}(d))$ 
and 
$h:= c_1(\mathcal O_{\mathbb P^n}(1))$.

\begin{remark}
It follows from Corollary \ref{corollary:degree_gamma_N-1_general} 
and \eqref{equation:Segre_classes_of_C_dual} that 
$d_{N-1}^{*} =(n+1)(d-1)^n$, 
which is positive. 
Therefore 
it follows from Lemma \ref{lemma:deg_X_m} that 
$\deg {X_{N-1}^*}=(n+1)(d-1)^n$ if $p = 0$: 
this recovers Boole's formula \eqref{equation:Boole's_formula}.
\end{remark}

\begin{proposition}\label{proposition:Schur_of_C_dual}
The Schur polynomial 
$\varDelta_{\lambda}(s(\mathcal N))$
in the Segre classes of 
the twisted normal bundle 
$\mathcal N={ N }_{v_{d}(\mathbb P^{n})/\mathbb P^N} \otimes \mathcal O_X(-1)$ 
for a partition 
$\lambda =(\lambda_1, \dots , \lambda_{e})$ 
is given by 
$$
\varDelta_{\lambda}(s(\mathcal N))
=
\frac{(d-1)^{\vert \lambda \vert}}{\vert \lambda \vert !}
f^{\lambda}
h^{\vert \lambda \vert}
\prod_{i=1}^{e}
\frac{(n+i)! }{(n+i-\lambda_i)! }
, 
$$
where 
$\vert \lambda \vert := \sum_{i=1}^{e} \lambda_i$, 
and 
$f^{\lambda}$ 
denotes the number of standard Young tableaux with shape $\lambda$. 
\end{proposition}

\begin{proof}
It follows from \eqref{equation:Segre_classes_of_C_dual} that 
{\allowdisplaybreaks 
\begin{align*}
\varDelta_{\lambda}(s(\mathcal N))
&= 
\det[s_{\lambda_i+j-i}(\mathcal N)]_{1 \le i,j\le {e}}
\\& =  
\det\begin{bmatrix}
\binom{n+1}{{\lambda_i+j-i}}
(d-1)^{\lambda_i+j-i} 
h^{\lambda_i+j-i} 
 \end{bmatrix}
_{1 \le i,j\le {e}}
\\& =  
(d-1)^{\vert \lambda \vert}
\det
\begin{bmatrix}
 \binom{n+1}{{\lambda_i+j-i}} 
 \end{bmatrix}
 _{1 \le i,j\le {e}}
 h^{\vert \lambda \vert}
.
\end{align*}
}%  
To evaluate the last determinant, 
we apply the method of 
Arbarello-Cornalba-Griffiths-Harris 
\cite[II, \S5, pp.94--95]{ACGH}: 
Performing the elementary column operations,
using the equality $\binom{n}{m-1}+\binom{n}{m}=\binom{n+1}{m}$, 
we obtain
{\allowdisplaybreaks 
\begin{align*}
\det  
\begin{bmatrix}
\binom{n+1}{\lambda_i+j-i}  
\end{bmatrix}_{1 \le i,j\le {e}} 
&
= %%%%%%%%%%%%%%%%%%%%%%%%%%%%%%%%%%
\det
\begin{bmatrix}
\binom{n+j}{\lambda_i+j-i}
\end{bmatrix}_{1 \le i , j \le e}
\\
&= %%%%%%%%%%%%%%%%%%%%%%%%%%%%%%%%%%
\frac{\prod_{j=1}^{e} (n+j)! }{\prod_{i=1}^{e} (n-\lambda_i+i)! }
\det 
\begin{bmatrix}
\frac{1}{(\lambda_{i}+j-i)!} 
\end{bmatrix}_{1 \le i,j \le {e}}
, 
\end{align*}
}%
and 
{\allowdisplaybreaks 
\begin{align*}
\det 
\begin{bmatrix}
\frac{1}{(\lambda_{i}+j-i)!} 
\end{bmatrix}_{1 \le i,j \le {e}}
&= %%%%%%%%%%%%%%%%%%%%%%%%%%%%%%%%%%
\frac{1} {\prod_{i=1}^e (\lambda_i +e-i )!} 
\det 
\begin{bmatrix}
\frac{ (\lambda_i +e-i )! }{(\lambda_{i}+j-i)!} 
\end{bmatrix}_{1 \le i,j \le {e}}
\\
&= %%%%%%%%%%%%%%%%%%%%%%%%%%%%%%%%%%
\frac{1}{\prod_{i=1}^e (\lambda_i +e-i )!} 
\det 
\begin{bmatrix}
(\lambda_i +e-i )^{e-j} 
\end{bmatrix}_{1 \le i,j \le {e}}
\\
&= %%%%%%%%%%%%%%%%%%%%%%%%%%%%%%%%%%
\frac{1}{\prod_{i=1}^e (\lambda_i +e-i )!} 
\prod_{1 \le i<j  \le {e}}(\lambda_i - \lambda_j -i+j) 
=
\frac{f^{\lambda}}{\vert \lambda \vert ! } 
, 
\end{align*}
}%  
where the last equality 
follows from \eqref{equation:hook_lengths}. 
\end{proof}

\begin{proof}[Proof of Theorem \ref{theorem:main_theorem}]
It follows from 
Theorem \ref{theorem:degree_gamma_m} 
\eqref{equation:degree_X_m_C_dual}
and 
Proposition \ref{proposition:Schur_of_C_dual} 
that 
{\allowdisplaybreaks 
\begin{align*}
d_{m}^{*} 
= 
\sum_{ \vert \lambda \vert = n} 
f^{\lambda + \varepsilon}
\int_{\mathbb P^n} 
\varDelta _{\lambda}(s(\mathcal N))   
&=
\frac{(d-1)^n}{n!}
\sum_{ \vert \lambda \vert = n} 
f^{\lambda}f^{\lambda + \varepsilon}
\prod_{i=1}^{N-m}
\frac{(n+i)! }{(n+i-\lambda_i)! }
.
\end{align*}
}%  
This completes the proof by Lemma \ref{lemma:deg_X_m}:
Indeed, the last term in the preceding formula is obviously a positive number.
\end{proof}

\begin{example}\label{example:Veronese_surface}
Consider the first non-trivial case in dimension $2$: 
$X=v_2(\mathbb P^2) \subseteq \mathbb P^5$ with $m=3$.  
According to Theorem \ref{theorem:main_theorem}, 
we have $\dim X_3^*=4$, 
and $\deg X_3^* = 21$ if $p=0$: 
Indeed, 
we have 
$\deg X_3^* = 3f^{(2,0)}f^{(3,1)} + 6f^{(1,1)}f^{(2,2)}$, 
where 
$f^{(2,0)}=f^{(1,1)}=1$, 
$f^{(3,1)}=3$ and $f^{(2,2)}=2$.
This means that 
\textit{%
there are exactly $21$ $3$-planes $W\subseteq \mathbb P^5$ 
such that 
$W$ is tangent to $X$ and 
intersects given $4$ lines in general position}: 
Indeed,  the class of hyperplane sections of $\mathbb G(m+1,V) = \mathbb G(3,6)$ 
is represented by a divisor 
$
\Omega(L)
$
with a linear subspace $L \subseteq \mathbb P^5$ of dimension $N-m-1=1$ 
(\cite[\S14.7]{fulton}). 
\end{example}

\begin{remark}
In terms of binomial and multinomial coefficients, 
one can write the formula in Theorem \ref{theorem:main_theorem}, 
as follows: 
\begin{align*}
\deg X_m^* 
=&
\frac{1}{(n+1)^{n}}
\deg X_{n}^{*} 
\sum_{ \vert \lambda \vert = n} 
{f^{\lambda}}
f^{\lambda + \varepsilon}
\frac{\prod_{i=1}^{N-m}\binom{n+i}{\lambda_i}}{\binom{n}{\lambda_1, \dots , \lambda_{N-m}}}
.
\end{align*}
\end{remark}

As a byproduct of our study of higher Gauss maps, 
we obtain a combinatorial identity, as follows:

\begin{corollary}\label{corollary:f_lambda_formula}
For any positive integer $n$, we have 
$$
\sum_{\vert \lambda \vert =n} (f^{\lambda})^{2} 
\prod_{1 \le i \le n} 
\frac{(n+i)!}{(n+i-\lambda_{i})!} 
=
(n+1)^{n} n! 
,
$$
where $\lambda = (\lambda_{1}, \dots, \lambda _{n} )$ is a partition of $n$. 
\end{corollary}

\begin{proof}
Set $m:=n$ in 
Theorem \ref{theorem:main_theorem}. 
\end{proof}

As another form of the degree formula in Theorem \ref{theorem:main_theorem}, 
we have

\begin{theorem}\label{theorem:main_theorem'}
Under the same assumption as in Theorem \ref{theorem:main_theorem}, 
the degree of $X_{m}^{*}$ 
in $\mathbb P^{\binom{N+1}{m+1}-1}$ is given by 
\begin{equation*}
\deg X_{m}^* 
=
\frac{1}{(n+1)^{n}} 
\deg X_{n}^{*}
\sum_{k=0}^{n}
\frac{
(-1)^{n-k}
(n+1)^{k}
}{
(n-k)!
}
\binom{M}{k}
\sum_{ \vert \lambda' \vert = n-k} 
f^{\lambda'}f^{\lambda' + \varepsilon'}
\prod_{i=1}^{m-n}
\frac{(n+i)! }{(n+i-\lambda'_i)! }
\end{equation*}
with $\deg X_{n}^{*} =(n+1)^{n}(d-1)^{n}$, 
where 
$M:=n+(N-m)(m-n)$, 
$\lambda' =(\lambda'_1, \dots , \lambda'_{m-n})$ is a partition 
with $\vert \lambda' \vert := \sum_{i=1}^{m-n}\lambda'_i$, 
$\varepsilon' := ((N-m)^{m-n})$ is $m-n$ copies of the integer $N-m$, 
and 
$f^{\lambda'}$ is the number of standard Young tableaux with shape $\lambda'$. 
\end{theorem}

\begin{proof}
Using 
Theorem \ref{theorem:degree_gamma_m} 
\eqref{equation:degree_X_m_C}
and 
Proposition \ref{proposition:Schur_of_C_dual}, 
we see that 
$d_{m}^{*}$ is equal to the right-hand side of the equality in the conclusion above. 
Since $d_{m}^{*} = \deg X_{m}^{*}$ 
by (the proof of) Theorem \ref{theorem:main_theorem}, 
the equality holds.  
\end{proof}

Compared with the formula in Theorem \ref{theorem:main_theorem}, 
the one in Theorem \ref{theorem:main_theorem'} seems more complicated.  
However, to compute $\deg X_{m}^{*}$ in practice, 
Theorem \ref{theorem:main_theorem'} is more efficient 
for the case when $m-n$ is small. For instance, we have

\begin{corollary}
Under the same assumption as in Theorem \ref{theorem:main_theorem}, 
for the case when $m = n+1$, we have 
$$
\deg X_{n+1}^{*} = 
\frac{1}{(n+1)^{n}} 
\deg X_{n}^{*}
\sum_{k=0}^{n} (-1)^{n-k}(n+1)^{k} \binom{N-1}{k}\binom{n+1}{n-k} 
.
$$
\end{corollary}

\begin{proof}
If $m=n+1$ in Theorem \ref{theorem:main_theorem'}, 
then 
$M=N-1$, 
$f^{\lambda'} =f^{\lambda'+\varepsilon'}=1$
for any $\lambda' = (\lambda'_{1})$, 
and 
$\lambda'_{1} = \vert \lambda'\vert = n-k$, 
hence the conclusion follows. 
\end{proof}

\begin{proof}[Proof of Corollary \ref{corollary:degree_approximation}]
It follows from 
\eqref{equation:hook_lengths}
that 
{\allowdisplaybreaks 
\begin{align*}
\frac{f^{\lambda+ \varepsilon}}{f^{\lambda}} 
&=
\frac{(n+\vert \varepsilon\vert)!}{n!}
\prod_{1 \le i \le N-m} 
\frac
{(N-m+\lambda_{i}-i)!}
{(N-n+\lambda_{i}-i)!}
,
\\
f^{\varepsilon}
&=
\vert \varepsilon\vert !
\prod_{1 \le i \le N-m} 
\frac{(N-m-i)!}{(N-n-i)!}
.
\end{align*}
}%
Therefore, 
using Theorem \ref{theorem:main_theorem} 
and \eqref{equation:degree_of_Grassmann_variety},  
we have 
{\allowdisplaybreaks 
\begin{align*}
&
\deg X_{m}^{*} 
=
\frac{\deg X_{n}^{*}}{n!(n+1)^{n}} 
\sum_{\vert\lambda\vert = n} 
(f^{\lambda})^{2}
\frac{(n+\vert \varepsilon\vert)!}{n!}
\prod_{1 \le i \le N-m} 
\frac
{(N-m+\lambda_{i}-i)!}
{(N-n+\lambda_{i}-i)!}
\frac{(n+i)!}{(n+i-\lambda_{i})!} 
,
\\ %%%%%%%%%%%%%%%%%%%%%%%%%%%%%%%%%%%%%%
& 
\tbinom{n+ \dim G}{n}
\deg G
\deg X_{n}^{*}
\\
&
\phantom{\deg X_{m}^{*}}
=
\tbinom{n+\vert \varepsilon\vert}{n}
f^{\varepsilon}
\frac{\deg X_{n}^{*}}{n!(n+1)^{n}} 
\sum_{\vert\lambda\vert = n} 
(f^{\lambda})^{2}
\prod_{1 \le i \le N-m} 
\frac{(n+i)!}{(n+i-\lambda_{i})!} 
\\
&
\phantom{\deg X_{m}^{*}}
=
\tbinom{n+\vert \varepsilon\vert}{n}
\vert \varepsilon\vert !
\frac{\deg X_{n}^{*}}{n!(n+1)^{n}} 
\sum_{\vert\lambda\vert = n} 
(f^{\lambda})^{2}
\prod_{1 \le i \le N-m} 
\frac{(N-m-i)!}{(N-n-i)!}
\frac{(n+i)!}{(n+i-\lambda_{i})!} 
,
\end{align*}
}%
where $G := \mathbb G(m-n,N-n)$.
Thus it follows that 
{\allowdisplaybreaks 
\begin{align*}
& 
\deg X_{m}^* 
\left/
\left[ 
\tbinom{n+ \dim G}{n}
\deg G
\deg X_{n}^{*}
\right]
\right. 
\\&=
\left.
\sum_{\vert\lambda\vert = n} 
(f^{\lambda} )^{2}
D(\lambda)
\prod_{1 \le i \le N-m} 
\frac{(n+i)!}{(n+i-\lambda_{i})!}
\right/
{
\sum_{\vert\lambda\vert = n} 
(f^{\lambda} )^{2}
\prod_{1 \le i \le N-m} 
\frac{(n+i)!}{(n+i-\lambda_{i})!} 
}
,
\end{align*}
}%
where we set 
$$
D(\lambda) := 
\prod_{1 \le i \le n}
\frac
{\tbinom{N-m+\lambda_{i}-i}{\lambda_{i}}}
{\tbinom{N-n+\lambda_{i}-i}{\lambda_{i}}}
=
\prod_{1 \le i \le n}
\prod_{1 \le l \le \lambda_{i}}
\frac
{N-m+l-i}
{N-n+l-i}
.
$$
The conclusion now follows from Lemma \ref{lemma:D_lambda} below, 
where
\begin{equation*}
D((1^{n})) = 
\tbinom{N-m}{n}\left/\tbinom{N-n}{n}\right.
, 
\quad 
D((n)) =
\tbinom{N-m+n-1}{n}\left/\tbinom{N-1}{n}\right.
. 
\qedhere
\end{equation*}
\end{proof}

\begin{lemma}\label{lemma:D_lambda}
For any partition $\lambda$ with $\vert \lambda \vert =n$, 
we have 
$$
D((1^{n})) \le D(\lambda) \le D((n))
.
$$
\end{lemma}

\begin{proof}
It suffices to show that for a partition $\lambda$, 
if $\lambda_{i} > \lambda_{j}$ with $i < j$, 
then 
$$
D((\dots, \lambda_{i} , \dots, \lambda_{j} , \dots))
> 
D((\dots, \lambda_{i}-1 , \dots, \lambda_{j}+1 , \dots))
,
$$
which is easily verified since we have $\lambda_{i}-i>(\lambda_{j}+1)-j$. 
\end{proof}

\begin{remark}
It follows immediately from Corollary \ref{corollary:degree_approximation} 
that 
for each $m$ with $n \le m \le N-1$, 
\begin{equation*}
\deg X_{m}^* 
\left/ 
\left[ 
\tbinom{n+ \dim G}{n}
\deg G 
\deg X_{n}^* 
\right]\right. 
\le 
(\tfrac{N-m+n-1}{N-1})^{n}
,
\end{equation*}
and for each $m$ with $n \le m \le N-n$, 
\begin{equation*}
(\tfrac{N-m-n+1}{N-2n+1})^{n}
\le
\deg X_{m}^* 
\left/ 
\left[ 
\tbinom{n+ \dim G}{n}
\deg G 
\deg X_{n}^* 
\right]\right. 
.
\end{equation*}
\end{remark}

\begin{conjecture}%[Cf.~Heier-Takayama {\cite[Theorem 1.1]{heier-takayama}}]
\label{conjecture:degree_approximation}
For the Veronese variety $X=v_d(\mathbb P^n)
\subseteq \mathbb P^{N}$ with 
$N+1 = \binom{n+d}{d}$ in characteristic $p=0$ 
and for each $m$ with $n \le m \le N-1$, 
we have 
\begin{equation*}
\deg X_{m}^* 
\le 
\left(
\frac{N-m}{N-n}
\right)^{n} 
\binom{n+ \dim G}{n}
\deg G 
\deg X_{n}^* 
, 
\end{equation*}
where $G:=\mathbb G(m-n,N-n)$ is 
the Grassmann variety parametrizing 
$(m-n)$-quotient spaces of an $(N-n)$-dimensional vector space. 
\end{conjecture}

\begin{observation}
\label{observation:virtual_decomposition}
\begin{enumerate}
\item %%%%%%%%%%%%%%%%%%%%%%%%%%%%%%%%%%%%
Numerical experiments 
by using {\tt Mathematica} 
suggest 
that the upper bound 
$\tbinom{N-m+n-1}{n}/\tbinom{N-1}{n}$
in Corollary \ref{corollary:degree_approximation}
could be replaced by 
$\big( \frac{N-m}{N-n} \big)^{n}$, 
as stated in Conjecture \ref{conjecture:degree_approximation}.
\item %%%%%%%%%%%%%%%%%%%%%%%%%%%%%%%%%%%%
Set $\alpha := \frac{(n+1)(d-1)}{N-n}\in \mathbb Q$. 
We have $c_{1}(\mathcal N)=(n+1)(d-1)h =(N-n)\alpha h$ 
with $\rk \mathcal N = N-n$, 
where  $h := c_{1}(\mathcal O_{\mathbb P^{n}}(1))$.
Suppose  
\begin{equation*}\label{equation:virtual_splitting}
\mathcal N \simeq \mathcal O_{\mathbb P^{n}}(\alpha)^{\oplus N-n}
\end{equation*}
{\it virtually}. 
Then, applying the argument in the proof of 
Theorem \ref{theorem:degree_formula_Veronese_curve} 
with $p=0$, 
we obtain 
\begin{equation*}\label{euqation:maximum_case}
\deg X_{m}^{*} = 
\left( 
\frac{N-m}{N-n}
\right)^{n}
\binom{n+ \dim G}{n}
\deg G
\deg X_{n}^* . 
\end{equation*}
Note that the right-hand-side is not necessarily an integer in general. 
On the other hand, 
according to Corollary \ref{corollary:degree_formula_in_dim1}, 
this `equality' actually holds true when $n=1$,  
even if $\mathcal N$ is not decomposed as above, 
as it is mentioned in Introduction.
This is because 
$\deg X_{m}^{*}$ is determined just by $c_{1}(\mathcal N)$ when $n=1$, 
which we see from the proof of Corollary \ref{corollary:degree_formula_in_dim1}. 

\end{enumerate}
\end{observation}

\subsection{Low dimensional cases}

\begin{theorem}\label{theorem:degree_formula_Veronese_surface}
Let $X = v_{d}(\mathbb P^{2}) \subseteq \mathbb P^{N}$ 
be the Veronese surface with 
$N+1 = \binom{d+2}{2}$ 
in characteristic $p=0$.  
For each $m$ with $n \le m \le N-1$, 
setting $e:= N-m$, we have 
{\allowdisplaybreaks 
\begin{align*}
\deg X_{m}^{*} 
&= 
\frac{e(3eN-N-5e-1)}{3(N-1)(N-2)(N-3)}
\\ & 
\hskip 96pt 
\binom{2+\dim \mathbb G(m-2,N-2)}{2}
\deg \mathbb G(m-2,N-2) \deg X_2^* 
\end{align*}
}%
with $\deg X_{2}^{*} = 3^{2}(d-1)^{2}$. 
\end{theorem}

\begin{proof}
It follows from Theorem \ref{theorem:main_theorem} that
\begin{equation*}
\deg X_{m}^{*} 
= 
\frac{1}{2! (2+1)^{2}}\deg X_{2}^{*} 
\sum_{\vert \lambda \vert = 2}
f^{\lambda}f^{\lambda+\varepsilon} 
\prod_{i=1}^{N-m} 
\frac{(2+i)!}{(2+i-\lambda_{i})!}
,
\end{equation*}
where $\varepsilon = ((m-2)^{N-m})$. 
In the summation above, 
$\lambda$ runs in 
$\{ (2)  , (1^{2})   \}$, and 
we have 
$f^{(2)}=f^{(1^{2})}=1$ (Example \ref{example:Veronese_surface}), 
where 
$(2):=(2,0^{N-m-1})$, 
$(1^{2}):= (1^{2},0^{N-m-2})$. 
Thus it follows that 
{\allowdisplaybreaks 
\begin{align*}
\deg X_{N-1}^{*} 
&
=
 \frac{1}{3}
 f^{(m)} 
\deg X_{2}^{*} 
,
\\
\deg X_{m}^{*} 
&= 
 \frac{1}{3}
\big( 
 f^{(m,(m-2)^{{e}-1})} + 
2  f^{((m-1)^{2},(m-2)^{{e}-2})}
\big)
\deg X_{2}^{*} 
,
\end{align*}
}%
where $e:= N-m$ with $2 \le m \le N-2$. 
It follows from \eqref{equation:hook_lengths} that 
{\allowdisplaybreaks 
\begin{align*}
f^{(m,(m-2)^{{e}-1})} &= 
\frac{({e}+1){e}}{(N-1)(N-2)}
\tbinom{M}{2} 
f^{\varepsilon} 
,
\\
f^{((m-1)^{2},(m-2)^{{e}-2})} &= 
\frac{{e}({e}-1)}{(N-2)(N-3)}
\tbinom{M}{2} 
f^{\varepsilon}
, 
\end{align*}
}%
where $M:=2+\vert \varepsilon \vert $. 
Therefore, for each $m$ with $n \le m \le N-1$ we have 
{\allowdisplaybreaks 
\begin{align*}
\deg X_{m}^{*} 
&= 
 \frac{1}{3}
 \left( 
 \frac{({e}+1){e}}{(N-1)(N-2)}
 + 
2 \frac{{e}({e}-1)}{(N-2)(N-3)}
\right)
\tbinom{M}{2} 
f^{\varepsilon}
\deg X_{2}^{*} 
\\&=
\frac{e(3eN-N-5e-1)}{3(N-1)(N-2)(N-3)}
\tbinom{M}{2} 
f^{\varepsilon}
\deg X_{2}^{*} 
.
\end{align*}
}%
The conclusion follows from 
Proposition \ref{proposition:ordinary_gauss_map}
\eqref{proposition:ordinary_gauss_map_Veronese} 
and 
\eqref{equation:degree_of_Grassmann_variety}.  
\end{proof}

\begin{theorem}\label{theorem:degree_formula_Veronese_3fold}
Let $X = v_{d}(\mathbb P^{3}) \subseteq \mathbb P^{N}$ 
be the Veronese $3$-fold with 
$N+1 = \binom{d+3}{3}$ 
in characteristic $p=0$. 
For each $m$ with $n \le m \le N-1$,
setting $e:= N-m$, we have 
{\allowdisplaybreaks 
\begin{align*}
\deg X_{m}^{*} 
&= 
\frac{e(
(8e^{2}-6e+1)N^{2} 
+(-42e^{2}+9e+6)N
+5(8e^{2}+3e+1)
)}
{8(N-1)(N-2)(N-3)(N-4)(N-5)}
\\ & 
\hskip 96pt 
\binom{3+\dim \mathbb G(m-3,N-3)}{3}
\deg \mathbb G(m-3,N-3) 
\deg X_3^* 
\end{align*}
}%
with $\deg X_{3}^{*} = 4^{3}(d-1)^{3}$. 
\end{theorem}

\begin{proof}
It follows from Theorem \ref{theorem:main_theorem} that
\begin{equation*}
\deg X_{m}^{*} 
= 
\frac{1}{3!(3+1)^{3}}\deg X_{3}^{*} 
\sum_{\vert \lambda \vert = 3} 
f^{\lambda}f^{\lambda+\varepsilon}
\prod_{i=1}^{N-m} 
\frac{(3+i)!}{(3+i-\lambda_{i})!}
,
\end{equation*}
where $\varepsilon = ((m-3)^{N-m})$. 
In the summation above, 
$\lambda$ runs in 
$\{ (3) , (2,1) , (1^{3})   \}$, and 
we have 
$f^{(3)}=f^{(1^{3})}=1$ and $f^{(2,1)}=2$, 
where $(3):=(3,0^{N-m-1})$, 
$(2,1):= (2,1,0^{N-m-2})$,  
$(1^{3}):= (1^{3},0^{N-m-3})$. 
Thus it follows that 
{\allowdisplaybreaks 
\begin{align*}
\deg X_{N-1}^{*} 
&=\frac{1}{16}
  f^{(m)}
\deg X_{3}^{*}
,
\\
\deg X_{N-2}^{*} 
&=\frac{1}{16}
\big( 
  f^{(m,m-3)} + 
5  f^{(m-1,m-2)} 
\big)
\deg X_{3}^{*}
,
\\
\deg X_{m}^{*} 
&=\frac{1}{16}
\big( 
  f^{(m,(m-3)^{{e}-1})} + 
5  f^{(m-1,m-2,(m-3)^{{e}-2})} +
5  f^{((m-2)^{3},(m-3)^{{e}-3})}
\big)
\deg X_{3}^{*}
,
\end{align*}
}%''
where $e:= N-m$ with $3 \le m \le N-3$. 
It follows from \eqref{equation:hook_lengths} that 
{\allowdisplaybreaks 
\begin{align*}
f^{(m,(m-3)^{{e}-1})} &= 
\frac{({e}+2)({e}+1){e}}{(N-1)(N-2)(N-3)}
\tbinom{M}{3} 
f^{\varepsilon} 
,
\\
f^{(m-1,m-2,(m-3)^{{e}-2})} &= 
\frac{2({e}+1){e}({e}-1)}{(N-2)(N-3)(N-4)}
\tbinom{M}{3} 
f^{\varepsilon} 
, 
\\
f^{((m-2)^{3},(m-3)^{{e}-3})} &= 
\frac{{e}({e}-1)({e}-2)}{(N-3)(N-4)(N-5)}
\tbinom{M}{3} 
f^{\varepsilon} 
, 
\end{align*}
}%
where $M:=3+\vert \varepsilon \vert $. 
Therefore, for each $m$ with $n \le m \le N-1$ we have 
{\allowdisplaybreaks 
\begin{align*}
& 
\deg X_{m}^{*}
\\&=\frac{1}{16}
\left( 
 \frac{({e}+2)({e}+1){e}}{(N-1)(N-2)(N-3)}
+\frac{10({e}+1){e}({e}-1)}{(N-2)(N-3)(N-4)}
+\frac{5{e}({e}-1)({e}-2)}{(N-3)(N-4)(N-5)}
\right)
\\ & 
\hskip 96pt 
\tbinom{M}{3} 
f^{\varepsilon} 
\deg X_{3}^{*}
\\&=
\frac{e((8e^{2}-6e+1)N^{2} 
+(-42e^{2}+9e+6)N
+5(8e^{2}+3e+1))}
{8(N-1)(N-2)(N-3)(N-4)(N-5)}
\tbinom{M}{3} 
f^{\varepsilon} 
\deg X_{3}^{*}
. 
\end{align*}
}%
The conclusion follows from 
Proposition \ref{proposition:ordinary_gauss_map}
\eqref{proposition:ordinary_gauss_map_Veronese} 
and 
\eqref{equation:degree_of_Grassmann_variety}. 
\end{proof}

\bigskip

\noindent%
\textit{Acknowlegments.} 
The author would like to thank Professor Fyodor L. Zak, 
who gave me detailed comments and expert advice. 
The author would like to thank Professor Shigeharu Takayama, too: 
the present work was started by his question on higher Gauss maps.  
The author wishes to thank Professor Satoru Fukasawa 
for useful comments and invaluable advice, 
and Professor Tomohide Terasoma for useful discussion. 
Finally the author would like to thank 
Professor Wu-yen Chuang,  
Professor Jiun-Cheng Chen,  
Professor Jungkai Chen 
and 
Professor Katsuhisa Furukawa 
for inviting me to the mini-conference on algebraic geometry (March 6, 2015) at 
National Center for Theoretical Sciences (NCTS), 
and for their warm hospitality throughout his stay in Taipei: 
In fact, the present work was done partly at NCTS. 

The author is supported by JSPS KAKENHI Grant Number 25400053 and 16K05113.

\end{document}